\documentclass[12pt]{amsart}
\usepackage[a4paper, left=3.1cm, right=3.1cm, bottom=3cm,top=3cm]{geometry}
\usepackage{mathscinet}
\usepackage{microtype}
\usepackage{amsmath, amsfonts, amssymb, amsthm}
\usepackage{xcolor}
\definecolor{ocean}{rgb}{0,0.5,0.5}
\definecolor{blue}{rgb}{0.00,0.26,0.50}
\usepackage[colorlinks = true, linkcolor=ocean, citecolor=blue]{hyperref}
\newtheorem{theorem}{Theorem}[section]
\newtheorem{proposition}[theorem]{Proposition}
\newtheorem{definition}[theorem]{Definition}
\newtheorem{remark}[theorem]{Remark}

\newtheorem{example}[theorem]{Example}
\newtheorem{question}[theorem]{Question}

\newtheorem{corollary}[theorem]{Corollary}
\newtheorem{lemma}[theorem]{Lemma}

\def\R{\mathbb R}  
\def\diag{{\rm diag}}

\def\Q{\mathbb Q}

\def\SL{\mathrm {SL}}

\def\<{\,<\!}
\def\>{\!>\,}

\def\R{\mathbb{R}}

\begin{document}
	
%\title[Uniform Approximate Lattices and Cartan subgroups in semisimple Lie group]{Uniform Approximate Lattices arising as a regular model set and Cartan subgroups in semisimple Lie group}

\title[Determining $\mathbb R$-Rank in Semisimple Lie Groups via uniform approximate Lattice ]{Determining $\mathbb R$-Rank in Semisimple Lie Groups via uniform approximate Lattice arising as  Regular Model Sets}
	
\author{Arunava Mandal and Shashank Vikram Singh}

\address{Department of Mathematics,
Indian Institute of Technology Roorkee,
Uttarakhand 247667, India}
\email{arunava@ma.iitr.ac.in}

\address{Department of Mathematical Sciences,
Indian Institute of Science Education and Research Mohali,
Punjab 140306, India}
\email{shashank@iisermohali.ac.in}

	\begin{abstract}
Let $G$ be a linear semisimple Lie group without compact factors. We show that uniform approximate lattices $\Lambda$ arising as regular model sets in $G$ determine the ambient group $G$ in a strong sense. Specifically, for every non-compact Cartan subgroup $C$ of $G$, there exists $g \in G$ such that the intersection $gCg^{-1} \cap \Lambda^2$
is non-empty and itself forms a uniform approximate lattice, extending a classical result of Mostow for lattices. The proof relies on a Moore-type ergodicity theorem for the hull of a strong approximate lattice, proved here as a key tool. Moreover, we prove that such approximate lattices determine the $\mathbb{R}$-rank of the ambient group $G$, drawing on ideas from the work of Prasad and Raghunathan on lattices. 
\end{abstract}

\maketitle
    \subjclass{\it 2020 Mathematics Subject Classification:} Primary 22E40, 22E15; Secondary 20N99, 22E25.

\keywords{\bf Keywords:} Uniform approximate lattice, Strong approximate lattice, Regular Model set, Cartan subgroup, Regular element, $\mathbb R$-Regular element, $\mathbb R$-rank of a semisimple Lie group.
	
\setcounter{tocdepth}{1}
\tableofcontents

\section{Introduction}
Approximate subgroups were introduced by Terence Tao \cite{Tao08} to extend ideas from additive combinatorics. A fundamental result of Breuillard, Green, and Tao \cite{BGT} shows that finite approximate subgroups admit a strong algebraic
structure. In the infinite setting, however, the class of approximate subgroups is too large to allow a comparable general structure theory. This has led to the study of more rigid subclasses, among which
\emph{uniform approximate lattices}, introduced by
Björklund and Hartnick \cite{bh18}, play a crucial role. Uniform approximate lattices are approximate subgroups of locally compact groups that are both uniformly discrete and relatively dense, and can be viewed as a non-commutative analogue of mathematical quasicrystals. In Euclidean spaces, they coincide with the classical Meyer sets, which arise from \emph{cut-and-project schemes}.
One of the central themes in the subject is to determine for which classes of groups approximate lattices in fact arise from such constructions. Some of the most significant breakthroughs in the subject have emerged from the study of such questions. Simon Machado extended Meyer’s original ideas to various classes of Lie groups, especially to nilpotent \cite{M20}, solvable \cite{M22}, and 
$S$-adic Lie groups \cite{M25S-adic}, while Hrushovski’s work \cite{h22}, from a geometrical point of view, provides deep structural results in the semisimple setting, providing an aperiodic analogue of Margulis’ arithmeticity theory, also proved independently by Machado \cite{M23} from an algebraic point of view.

In this article, our aim is to study the other way: how uniform approximate lattices reflect structural properties of linear semisimple Lie groups (without compact factors), motivated by the classical results of Mostow \cite{M70} on uniform lattices and of Prasad--Raghunathan \cite{pr} on lattices and Cartan subgroups in semisimple Lie groups. In particular, our main goal is to understand how the presence of a uniform approximate lattice determines the \emph{rank} of such groups. To approach this problem, we study the interaction between Cartan subgroups and uniform approximate lattices that arise as regular model sets in linear semisimple Lie groups. An approximate lattice that arises as a regular model set is automatically a \emph{strong approximate lattice}, a notion different from that of a uniform approximate lattice. Strong approximate lattices are defined via the existence of invariant probability measures on the hull under the natural group action. Using ergodic theory and homogeneous dynamics, we study the interaction of regular elements and Cartan subgroups with strong approximate lattices that arise as regular model sets; some of these results may also be of independent interest.

\subsection{Moore’s Ergodicity type theorem for hull of a strong approximate lattice}
Strong approximate lattices were introduced by Björklund and Hartnick \cite{bh18} as a
dynamical strengthening of approximate lattices, defined via the existence
of a $G$-invariant probability measure on the hull of the set. A large and
important class of examples arises as \emph{model sets} over
cut-and-project schemes.
More precisely, let $(G,H,\Gamma)$ be a cut-and-project scheme, where $G$
and $H$ are locally compact groups and $\Gamma$ is a lattice in $G\times H$
such that the projection of $\Gamma$ to $G$ is injective and its projection
to $H$ is dense. Given a compact subset $W\subset H$, called the
\emph{window}, the associated \emph{model set} is defined by
\[
\Lambda=\{\pi_G(\gamma): \gamma\in\Gamma,\ \pi_H(\gamma)\in W\},
\]
where $\pi_G$ and $\pi_H$ denote the natural projections from $G\times H$
onto $G$ and $H$, respectively. Equivalently, one may write
$\Lambda=P_0(\Gamma,W)$. When the window $W$ satisfies suitable regularity
conditions, the resulting set is called a \emph{regular model set} (see §2.3 for details). Such
sets are known to produce strong approximate lattices and form a central
class of examples in the theory. 

Moore's ergodicity theorem is a fundamental result in homogeneous dynamics.
It states that if $\Gamma$ is an irreducible lattice in a semisimple Lie group
$G$ without compact factors, then the action of any subgroup of $G$ with
non-compact closure on the homogeneous space $G/\Gamma$ is ergodic with
respect to the $G$-invariant probability measure. In the setting of strong
approximate lattices, a natural analogue arises through the study of the
\emph{hull} of the approximate lattice, defined as the closure of
$\{g\Lambda : g \in G\}$
in the Chabauty--Fell topology. When a strong approximate lattice arises
as a regular model set over a cut-and-project scheme, its hull admits a
unique $G$-invariant probability measure. This allows one to study the
dynamical system given by the $G$-action on the hull in a way that parallels the classical action of $G$ on $G/\Gamma$, and we obtain Moore’s Ergodicity-type theorem for the hull of a strong approximate lattice in Theorem \ref{theorem-moore’s ergodicity-type theorem for hull}. Theorem \ref{theorem-moore’s ergodicity-type theorem for hull} plays a crucial role in our work. It provides a key tool for establishing a relation between Cartan subgroups and strong approximate lattices (see Theorem \ref{theorem-non-compact-cartan}).

\begin{theorem}[Moore’s Ergodicity-type theorem for hull of a strong approximate lattice]\label{theorem-moore’s ergodicity-type theorem for hull}
    Let $G$ be a linear semisimple Lie group without a compact factor and let $\Lambda = P_0(\Gamma, W_0)$ be a regular model set over a cut-and-project scheme $(G,H,\Gamma)$, where $H$ is a linear semisimple Lie group without a compact factor and $\Gamma$ is an irreducible lattice in $G\times H$. Let $\Omega_{\Lambda}$ be the hull of $\Lambda$. 
    If $M$ is a non-compact subgroup of $G$, then it acts ergodically on the hull $\Omega_{\Lambda}$.
\end{theorem}

\subsection{Strong approximate lattice and regular elements}
Let $G$ be a connected semisimple linear Lie group without compact factors. 
A classical theorem of Prasad and Raghunathan asserts that if $\Gamma$ is a lattice in $G$, 
then $\Gamma$ contains a \emph{regular element}. Recall that an element $g \in G$ is called 
\emph{regular} if its centralizer in $G$ has minimal dimension; equivalently, $g$ lies in a 
unique Cartan subgroup of $G$. Regular elements play a fundamental role in the structure 
theory of semisimple groups and in the study of lattices, as they detect Cartan 
subgroups and provide important information about the algebraic and dynamical properties 
of the group.
In Theorem \ref{theorem-regular element-strong approximate lattice}, we extend the theorem of Prasad and Raghunathan from lattices to the broader 
setting of \emph{strong approximate lattices}. More precisely, we prove that if 
$\Lambda \subset G$ is a strong approximate lattice, then $\Lambda^2$ contains a regular 
element. 
%Our proof relies on dynamical properties of the hull associated with a strong approximate lattice. 
We will use Theorem \ref{theorem-regular element-strong approximate lattice} in the proof of Theorem \ref{theorem-non-compact-cartan} and in later sections.

\begin{theorem}\label{theorem-regular element-strong approximate lattice}
 Let $G$ be a linear semisimple Lie group without a compact factor. Let $\Lambda$ be an irreducible strong approximate Lattice in $G$. Then $\Lambda^2$ contains a regular element.   
\end{theorem}

To prove Theorem \ref{theorem-regular element-strong approximate lattice}, a key ingredient is 
Lemma \ref{lemma-two topology}, which can be interpreted as a recurrence phenomenon detected 
through local neighborhoods. In the classical setting, when $G$ is a connected Lie group and 
$\Gamma$ is a lattice in $G$, the group $G$ acts on the homogeneous space $G/\Gamma$. 
In our situation, when $\Lambda$ is a strong approximate lattice, the group $G$ instead acts 
on the hull $\Omega_\Lambda$, endowed with the Chabauty--Fell topology. By making use of this 
topology on the hull together with Poincaré recurrence, we show in 
Lemma \ref{lemma-two topology} that for every $a \in G$ and every identity neighborhood 
$U \subset G$, one has
$Ua^lU \cap \Lambda \neq \varnothing$ for infinitely many integers $l$. This lemma is the crucial step in the proof.

\subsection{Strong approximate lattice and non-compact Cartan subgroups}
A classical result of Prasad and Raghunathan \cite{pr} states that, in a connected semisimple Lie group $G$,  suitable conjugates of a Cartan subgroup $C$ intersect a lattice $\Gamma$, and in fact contain a regular element. Moreover, $gCg^{-1}\cap\Gamma$ is uniform lattice in $gCg^{-1}$; if $\Gamma$ is uniform, it was proved by Mostow.
A Cartan subgroup $C$ of a connected Lie group $G$ is a maximal nilpotent subgroup $C$ of $G$ with the property that 
if $L$ is any closed normal subgroup of finite index in $C$, then $L$ has finite index in 
its normalizer in $G$. For algebraic groups, Cartan subgroups are the centralizers of their maximal tori. 
For a semisimple algebraic group, they coincide with maximal tori. For more detailed and recent results on the structure of Cartan subgroups in Lie groups, we refer \cite{MS21}, \cite{MS26}, \cite{KMS26}, and their applications see \cite{DM17}, \cite{BM18}, \cite{KM26}, for instance.

In Theorem \ref{theorem-non-compact-cartan}, we establish an analogue of the result stated above for strong approximate lattices arising as regular model sets. More precisely, we show that every non-compact Cartan subgroup of a linear semisimple Lie group without compact factors intersects a strong approximate lattice and, moreover, contains a regular element.
Our proof relies crucially on ergodic techniques for the $G$-action on the hull of a strong approximate lattice, described in Theorem \ref{theorem-moore’s ergodicity-type theorem for hull}. This requirement motivates the additional assumption that the strong approximate lattice arises as a regular model set. The argument of Prasad and Raghunathan inspires the general strategy; however, in our setting, it is necessary to relate two different topologies: the Chabauty–Fell topology on the hull and the usual topology on $G$. 
%(cf. Lemma \ref{lemma-two topology}).

%Results of this type provide further evidence that strong approximate lattices share deep structural features with lattices, and they contribute to the growing parallel between classical rigidity phenomena in Lie groups and their counterparts in the theory of approximate lattices.

\begin{theorem}\label{theorem-non-compact-cartan}
  Let $G$ be a linear semisimple Lie group without a compact factor and let $\Lambda = P_0(\Gamma, W_0)$ be a regular model set over a cut-and-project scheme $(G,H,\Gamma)$, where $H$ is a linear semisimple Lie group without a compact factor and $\Gamma$ is an irreducible lattice in $G\times H$. Let $C$ be a non-compact Cartan of $G$ and $C^0$ be its connected component of the identity.
   Then we have the following:
   \begin{enumerate}
       \item $\Lambda^2\cap gC^0g^{-1}\neq \emptyset$ for some $g\in G$.

       \item If $\Gamma$ is uniform, then $\Lambda^2 \cap gCg^{-1}$ is an uniform approximate lattice in $gCg^{-1}$.

        \item If $\Gamma$ is uniform, then $\Lambda^2 \cap gC^0g^{-1}$ is an uniform approximate lattice in $gC^0g^{-1}$. In particular, $\Lambda^2 \cap gC^0g^{-1}$ is a Meyer set.

   \end{enumerate}
  \end{theorem}

   The
$\mathbb{R}$-\emph{rank} of a semisimple Lie group $G$ is defined as the dimension of a
maximal $\mathbb{R}$-split torus in $G$. A theorem of Prasad and Raghunathan shows
that if $\Gamma$ is a lattice in $G$, then $\Gamma$ contains a free abelian subgroup
of rank $r=\mathbb{R}\text{-rank}(G)$ \cite{pr}. Thus the lattice necessarily contains a
subgroup isomorphic to $\mathbb{Z}^r$, where $r$ is determined by the $\mathbb{R}$-rank
of the ambient group. It is natural to ask whether an analogous structural property
holds for approximate lattices. The following result shows that this is indeed the
case for uniform approximate lattices arising as regular model sets.

\begin{corollary}\label{cor-free-abelian-group}
  Let $G$ be a linear semisimple Lie group without compact factors. Let $\Lambda$ be a
uniform approximate lattice arising as a regular model set. Then $\langle \Lambda^2
\rangle$ contains a free abelian subgroup of rank
$r=\mathbb{R}\text{-rank}(G)$.  
\end{corollary}

For $\Lambda$ and $G$ as in Theorem \ref{theorem-non-compact-cartan}, by using ergodicity, we show that $\Lambda^2$ contains  $\mathbb{R}$-regular elements in Proposition \ref{proposition-R-regular-element}, in a similar spirit to Theorem \ref{theorem-regular element-strong approximate lattice}.

\subsection{Characterization of $\mathbb R$-rank of linear semisimple group by Uniform approximate Lattice}
In the classical theory of lattices in semisimple Lie groups, a central theme is that discrete subgroups reflect structural properties of the ambient group. Let $G$ be a semisimple Lie group and let $\Gamma \subset G$ be a lattice. The lattice $\Gamma$ is called \emph{uniform} if $G/\Gamma$ is compact.
A fundamental result due to Prasad and Raghunathan shows that a lattice determines the ambient group in a very strong sense. More precisely, if two connected semisimple Lie groups contain isomorphic lattices, then their $\mathbb{R}$-ranks coincide \cite{pr}. In the uniform case, this result is due to J. A. Wolf \cite{W62}.

In the following theorem, we prove that uniform approximate lattices arising as regular model sets in linear semisimple Lie groups without compact factors determine the $\mathbb{R}$-rank of the ambient group.

\begin{theorem}\label{theorem-determine-rank}
Let $G_1$ and $G_2$ be connected linear semisimple Lie groups without
compact factors. Let $\Lambda_1 \subset G_1$ and $\Lambda_2 \subset G_2$
be uniform approximate lattices arising as regular model sets. 
If the groups $\langle \Lambda_1^2 \rangle$ and $\langle \Lambda_2^2 \rangle$
are isomorphic (as a topological group), then
$\mathbb{R}\text{-rank}(G_1)=\mathbb{R}\text{-rank}(G_2).$
\end{theorem}

\begin{remark}
  If the groups $\langle \Lambda_1^2 \rangle$ and $\langle \Lambda_2^2 \rangle$, as defined in Theorem~\ref{theorem-determine-rank}, are discrete and isomorphic as abstract groups, then it follows that $\mathbb{R}\text{-rank}(G_1) = \mathbb{R}\text{-rank}(G_2)$.
\end{remark}

The paper is organized as follows. In Section 2, we fix notation, recall key definitions, and provide examples of strong approximate lattices arising as regular model sets. In Section 3, we prove Moore’s ergodicity-type theorem for the hull of a strong approximate lattice in Theorem \ref{theorem-moore’s ergodicity-type theorem for hull}. In Section 4, we establish a result on the interaction between regular elements and strong approximate lattices in Theorem \ref{theorem-regular element-strong approximate lattice}. In Section 5, assuming the strong approximate lattice arises as a regular model set, we prove a strong interaction with any non-compact Cartan subgroup in Theorem \ref{theorem-non-compact-cartan}, Corollary \ref{cor-free-abelian-group} and Proposition \ref{proposition-R-regular-element}. Finally, in Section 6, we prove Theorem \ref{theorem-RR} for a characterisation of $\mathbb R$-rank of a semisimple Lie group, and deduce Theorem \ref{theorem-determine-rank}.

\section{Preliminaries}
In this section, we review basic definitions and provide explicit examples and results that will be used throughout the paper. Unless otherwise specified, all Lie groups considered are connected.

\subsection{Uniform approximate lattices}

Let $G$ be a locally compact, second-countable group. Let $A,B$ be subsets of $G$, then we denote by $AB=\{ab \in G:a \in A, b\in B $ and $A^{-1}=\{a^{-1}\in G : a \in A \}$, their product and inverse set, respectively.

\begin{definition}[Approximate subgroup]

An approximate subgroup of a group $G$ is a subset $\Lambda \subset G$ such
that $\Lambda$ is symmetric (i.e.\ $\Lambda = \Lambda^{-1}$), contains the identity
element $e \in G$, and there exists a finite subset $F \subset G$ satisfying
\[
\Lambda^2 := \{ \lambda_1 \lambda_2 \in G \mid \lambda_1, \lambda_2 \in \Lambda \}
\subset F \Lambda .
\] 

\end{definition}

We recall that a subset $P \subset G$ is called {\it relatively dense} if there exists a compact subset $K \subset G$ such that $G=KP$. Further, $P$ is called uniformly discrete if the identity $e \in G$ is not an accumulation point of $P^{-1}P$, i.e.\ $P P^{-1} \cap U = \{e\}$ for some neighbourhood $U$ of the identity.

\begin{definition}[Uniform approximate lattice]
    An approximate subgroup $\Lambda$ of a group $G$ is called a \emph{uniform approximate lattice} if it is uniformly discrete and relatively dense.
\end{definition}

\begin{definition}[Approximate lattice]
    A discrete approximate subgroup $\Lambda$ of a group $G$ is called an \emph{approximate lattice} if there exists a subset $F \subset G$ with finite Haar-measure such that $G=FP$.
\end{definition}

\subsection{Chabauty-Fell Topology and Strong Approximate Lattices}
In this section, recall Chabauty-Fell Topology and some basis definitions which are relevant for this article, see for example \cite{bhs19}, and \cite{bh18}. Let $G$ be a locally compact second countable group and let $\mathcal{C}(G)$ denote
the set of closed subsets of $G$. The \emph{Chabauty--Fell topology} on $\mathcal{C}(G)$
is defined by the subbasis of open sets
\[
\mathcal{O}_U := \{ Q \in \mathcal{C}(G) \mid Q \cap U \neq \varnothing \}
\quad\text{and}\quad
\mathcal{O}_K := \{ Q \in \mathcal{C}(G) \mid Q \cap K = \varnothing \},
\]
where $U \subset G$ ranges over open subsets and $K \subset G$ over compact subsets.

We recall here that the map
\[
G \times \mathcal{C}(G) \longrightarrow \mathcal{C}(G),
\qquad
(g,Q) \longmapsto gQ,
\]
defines a continuous action of $G$ on $\mathcal{C}(G)$, and that $\mathcal{C}(G)$
is a compact metrizable space (see~\cite[Propositions 1.7 and 1.8]{f07}).
Convergence in the Chabauty--Fell topology can be characterised as follows:
a sequence $(Q_i)_{i \ge 0}$ in $\mathcal{C}(G)$ converges to $Q \in \mathcal{C}(G)$
if and only if the following two conditions hold:
\begin{enumerate}
  \item for every $x \in Q$ there exist $x_i \in Q_i$ such that $x_i \to x$ as
  $i \to \infty$;
  \item if $x_i \in Q_i$ for all $i \in \mathbb{N}$, then every accumulation point
  of the sequence $(x_i)_{i \ge 0}$ belongs to $Q$.
\end{enumerate}

Given a closed subset $Q \subset G$, its $G$-orbit in $\mathcal{C}(G)$ is the set
$G \cdot Q := \{ gQ \mid g \in G\},$
where
$gQ := \{\, gq \mid q \in Q\}.$
The invariant hull $\Omega_Q$ is defined as the closure of this orbit in
$\mathcal{C}(G)$ with respect to the Chabauty--Fell topology, that is,
\[
\Omega_Q := \overline{\{\, gQ \mid g \in G \,\}} \subset \mathcal{C}(G).
\]

Note that if $H \leq G$ is a closed subgroup,
then $\Omega_H$ is $G$-equivariantly homeomorphic to the one-point compactification
of $G/H$ (see~\cite[Lemma~6]{bhs19}).

\begin{definition}[Strong approximate lattice] %\cite{bh18}
 Let $\Lambda$ be a discrete approximate subgroup of a locally compact second countable group $G$. We say that $\Lambda$ is a \emph{strong approximate lattice} if there exists a $G$-invariant Borel probability measure $\nu$ on $\Omega_\Lambda$
  such that $\nu(\{\varnothing\}) = 0$ (in which case we say that $\nu$ is
  \emph{proper}).

\end{definition}

\subsection{Model set}
Model sets were introduced in the seminal work of Meyer (1970). The model set is used to describe special point patterns, especially in the study of quasicrystals. A \textit{model set} is a set of points obtained by projecting a higher-dimensional lattice onto a lower-dimensional space and selecting points using a window. These sets help explain non-periodic but highly ordered structures in mathematics and physics. To define model sets formally, we use the \textit{cut-and-project scheme}. We recall definitions from Subsection 2.3 of \cite{bh18}.

 \begin{definition}[Cut-and-project scheme]
A cut-and-project scheme is a triple $(G,H,\Gamma)$ where $G$ and $H$ are locally compact second countable (lcsc) groups and $\Gamma < G \times H$ is a lattice which projects injectively to $G$ and densely to $H$. A cut-and-project scheme is called uniform if $\Gamma$ is a uniform lattice.

\end{definition}

Given a cut-and-project scheme $(G,H,\Gamma)$ we denote by $\pi_G,\pi_H$ the coordinate projections
of $G \times H$ on $G$ and $H$, respectively and set
$\Gamma_G := \pi_G(\Gamma),\;\Gamma_H := \pi_H(\Gamma).$
We then define a map $\tau \colon \Gamma_G \to H$ by
$\tau := \pi_H \circ (\pi_G|_\Gamma)^{-1}.$
Note that the image of $\tau$ is precisely $\Gamma_H$. If $G$ and $H$ are abelian, then $\tau$ is known as \emph{$\ast$-map}.

\begin{definition}[Regular model set]
Let $(G,H,\Gamma)$ be a cut-and-project scheme with associated $\tau$-map
$\tau \colon \Gamma_G \to H$. Given a compact subset $W_0 \subset H$, the pre-image
$P_0(\Gamma,W_0) := \tau^{-1}(W_0) \subset G$
is called a \emph{weak model set}, and $W_0$ is called its \emph{window}. A weak model set is called
a \emph{model set} if its window has non-empty interior. It is called \emph{regular} if the window
$W_0$ is Jordan-measurable (i.e the boundary of $W_0$, $\partial W_0$, has Haar measure zero) with dense interior, aperiodic (i.e.\ $\operatorname{Stab}_H(W_0)=\{e\}$),
and satisfies
$\partial W_0 \cap \Gamma_H = \varnothing.$
 %A \emph{Meyer set} is a relatively dense subset of a uniform model set.
\end{definition}

\begin{definition}[Uniform model set]
    A model set is called a \emph{uniform model set} if the underlying cut-and-project scheme is uniform.
\end{definition}

\begin{definition}[Meyer set]
    A relatively dense subset of a model set is called \emph{Meyer Set}.
\end{definition}

Now we recall \cite[Proposition 2.13]{bh18}, which states that under suitable assumptions on the window, uniform model sets are uniform approximate lattices.

\begin{proposition} \label{proposition-uniformly discrete-app latices from model set}
Let $\Lambda = P_0(\Gamma, W_0)$ be a model set over $(G,H,\Gamma)$.
\begin{enumerate}
\item $\Lambda$, and in fact $\Lambda^{-1}\Lambda$, is uniformly discrete. % In particular, $\Lambda$ has finite local complexity.
\item $\Lambda$ is an approximate subgroup if its window
is symmetric and contains the identity.
 \item If $\Gamma$ is uniform, then $\Lambda$ is relatively dense.
 \item If $\Gamma$ is non-uniform, then $\Lambda$ is not relatively dense.
\end{enumerate}
In particular, a uniform model set is a uniform approximate lattice if its window is symmetric and
contains the identity.
\end{proposition}

The following result (see \cite[corollary 3.5]{BHP}) gives a way to construct strong approximate lattices from a regular model set.

\begin{corollary}[Regular model sets as strong approximate lattices]\label{corollary-Regular model sets as strong approximate lattices}
A regular model set is a strong approximate lattice provided its window is symmetric and contains
the identity.
\end{corollary}

It is natural to ask how one can construct regular windows. The following example illustrates a method to obtain a regular window, in particular, producing a strong approximate lattice.

    \begin{example}\label{example-strong app.lattice}
Consider the cut-and-project scheme
\[
(G,H,\Gamma)
=
\bigl(\mathrm{SL}_2(\mathbb{R}),\, \mathrm{SL}_2(\mathbb{R}),\,
\mathrm{SL}_2(\mathbb{Z}[\sqrt{2}])\bigr),
\]
where $\Gamma$ is embedded diagonally via the two real Galois embeddings
of $\mathbb{Q}(\sqrt{2})$.
Let $\tau \colon \Gamma_G \to H$ be the associated $\tau$-map.
Fix $\varepsilon > 0$ sufficiently small and define
$W_0 := \overline{B_\varepsilon(e)} \subset H,$
where $B_\varepsilon(e)$ denotes the open $\varepsilon$-ball around the identity
with respect to a left-invariant Riemannian metric on $\mathrm{SL}_2(\mathbb{R})$.
\end{example}

\begin{proposition}
The set $W_0$ is a regular window and the model set $\Lambda := P_0(\Gamma,W_0) = \tau^{-1}(W_0) \subset \mathrm{SL}_2(\mathbb{R})$ is a strong approximate lattice.
\end{proposition}

\begin{proof}
We verify the defining properties of a regular window. The set $\Gamma_H=\pi_H(\Gamma)$ is countable. Consider the set $\mathcal R := \{ d(e,h) : h \in \pi_H(\Gamma) \} \subset \mathbb R_{\ge 0}$; since $\pi_H(\Gamma)$ is countable, $\mathcal R$ is countable, hence $\mathbb R_{>0} \setminus \mathcal R \neq \varnothing$.
Choosing $\varepsilon$ from $\mathbb R_{>0} \setminus \mathcal R \neq \varnothing$ ensures $\partial W_0 \cap \Gamma_H =\varnothing.$
By construction, $W_0$ is compact and has non-empty interior $\operatorname{int}(W_0)=B_\varepsilon(e)$.
Since $B_\varepsilon(e)$ is symmetric, so is $W_0$, and clearly $e \in W_0$. The boundary $\partial W_0$ is a smooth codimension-one submanifold of
$\mathrm{SL}_2(\mathbb{R})$, hence has Haar measure zero. Therefore, $W_0$ is Jordan-measurable.
To show periodicity of $W_0$, let $h \in \mathrm{SL}_2(\mathbb{R})$ be such that $hW_0=W_0$. Then $h$ fixes the identity and preserves a sufficiently small neighborhood of it. For $\varepsilon$
small enough, this implies $h=e$ (Since $hW_0=W_0$ implies that $\overline{B_\varepsilon(e)}=\overline{B_\varepsilon(h)}$, locally these are two balls in the Euclidean domain, therefore centers will be the same). Hence
$\operatorname{Stab}_H(W_0)=\{e\}.$
Thus, $W_0$ is a regular window. Since the window is symmetric and contains the
identity, Proposition \ref{proposition-uniformly discrete-app latices from model set} and Corollary \ref{corollary-Regular model sets as strong approximate lattices}  imply that $\Lambda$ is a
strong approximate lattice in $\mathrm{SL}_2(\mathbb{R})$.
\end{proof}

\subsection{Regular and $\mathbb R$-regular elements}
Let $\bf{G}$ be a connected semisimple algebraic group defined over $\mathbb R$ and contained in $\mathrm{GL}(n,\mathbb{C})$ for some $n$. We assume that if $g \in \bf G$ then conjugate transpose of $g$, $\Bar{g}^t$, also belongs to $\bf{G}$, that is, $\bf{G}$ is self-adjoint. Let $\bf{G}_{\mathbb R}:= \bf{G} \cap \mathrm{GL}(n,\mathbb{R})$ be group of real points in $\bf G$. The connected component of the identity, $G$, of $\bf{G}_{\mathbb R}$ is a linear semisimple group. Moreover, any connected linear semisimple group can be realized in this way (see \cite{M55}).
Let $G$ be a connected semisimple Lie group defined over $\R$. Consider the adjoint representation $\mathrm{Ad}$ of $G$ on its Lie algebra $\mathfrak{g}$. Let $\sigma$ be the Cartan involution of $\mathfrak{g}$. Then 
    $\mathfrak{g}=\mathfrak{k}+ \mathfrak{p},$ where $\mathfrak{k}=\{X\in \mathfrak{g}: \sigma X=X\}$ and $\mathfrak{p}=\{X\in \mathfrak{g}: \sigma X=-X\}$ satisfying $[\mathfrak{k},\mathfrak{k}]\subset\mathfrak{k}$, $[\mathfrak{k},\mathfrak{p}]\subset \mathfrak{p}$ and $[\mathfrak{p},\mathfrak{p}] \subset \mathfrak{k}$. Any two Cartan decompositions are conjugate to each other under the action of $\mathrm{Ad}(G)$, where $\mathrm{Ad}(G)$ denotes the adjoint group of $\mathfrak{g}$. Let $K$ be the analytic subgroup of $ G$ with Lie algebra $\mathfrak{k}$. We fix a maximal abelian subalgebra (Cartan subspace) $\mathfrak{a} \subset \mathfrak{p}$. It is well known that any two Cartan subspaces of $\mathfrak{p}$ are conjugate under $K$. Note that the Lie algebra of the subgroup of real points in a maximal $\mathbb{R}$-split torus in $\bf{G}$ can be conjugated to a Cartan-subspace of $\mathfrak{p}$. The dimension of $\mathfrak{a}$ is the \emph{$\mathbb{R}$-rank} of $G$.

    For $g\in G$, let $n_1(g)$ denote the multiplicity with which 1 occurs as an eigenvalue of $\mathrm{Ad}(g)$ and $m_1(g)$ denote the number of eigenvalues (counted with multiplicity) of modulus 1 of $\mathrm{Ad}(g)$.

\begin{definition}[Regular and $\mathbb R$-regular elements] An element $g \in G$ is regular (resp. $\mathbb{R}$-regular) if for all $h \in G$ we have $n_1(g) \leq n_1(h)$ (resp. $m_1(g) \leq m_1(h)$) and $-1$ is not an eigen value of $\mathrm{Ad}(g)$.
    
\end{definition}

It is known that any regular element in $G$ is contained in a unique Cartan subgroup. A subalgebra $\mathfrak c$ is said to be a Cartan subalgebra of a Lie algebra $\mathfrak g$ if $\mathfrak c$ is a nilpotent Lie algebra, and its normalizer $N_\mathfrak g(\mathfrak c)=\mathfrak c$. Lie algebra of a Cartan subgroup is a Cartan subalgebra, and $\exp(\mathfrak c)=C^0$. Moreover, there is a one-to-one correspondence between the Cartan subalgebras of Lie algebras and the Cartan subgroups of Lie groups.

\section{Ergodic action on hull of a strong approximate lattice arising as  Regular Model Sets}
In this section, we prove Theorem \ref{theorem-moore’s ergodicity-type theorem for hull}, and we first recall Moore's Ergodicity \cite[Theorem 2.2.6]{Z84}.

\begin{theorem}[Moore's Ergodicity Theorem]\label{moores ergodicity}
Let $G = G_1 \times \cdots \times G_n$ be a direct product of connected,
simple, non-compact Lie groups with finite center, and let $\Gamma \le G$
be an irreducible lattice. If $H \le G$ is a closed subgroup and $H$
is not compact, then $H$ acts ergodically on $G/\Gamma$.
\end{theorem}

Now, we also recall some parts of \cite[Theorem 3.1]{BHP} for our purpose.

\begin{theorem}\label{uniquemeasueonhull}
Let $G$ be a connected locally compact second countable group, and let $\Lambda = P_0(\Gamma, W_0)$ be a regular model set over a cut-and-project scheme $(G,H,\Gamma)$. Denote $Y:=\frac{G \times H}{\Gamma} $ and the punctured hull $\Omega_{\Lambda}^\times:=\Omega_{\Lambda}\setminus \{\phi\}$. There exists a unique $G$-equivariant Borel map $\beta: \Omega_{\Lambda}^\times \to Y $ which maps $\Lambda$ to $(e,e)\Gamma$ and has a closed graph. This map has the following additional properties.
\begin{enumerate}
    \item If $Y^{ns}:= \{(g,h)\Gamma \in Y : h^{-1}W_0 \text{ is } \Gamma\text{-}regular\}$ and $X^{ns} := \beta^{-1}(Y^{ns})$, then
$\beta|_{X^{ns}} : X^{ns} \to Y^{ns}$ is bijective. 

\item If $\Gamma$ is co-compact, then $\beta$ is continuous.
\end{enumerate}
\end{theorem}

It is well known that the space $Y$ has a unique $G$-invariant measure, say $\mu$. Using the above result, in \cite{BHP}, authors have defined a $G$-invariant probability measure on $\Omega_{\Lambda}^\times$ by $\nu:= (\beta|_{X^{ns}})^{-1}_*\mu|_{Y^{ns}}.$ By definition, we have for every borel subset $A\subset \Omega_{\Lambda}^\times$, $\nu(A):=\mu(\beta(A \cap X^{ns}))=\mu(\beta(A)\cap Y^{ns})=\mu(\beta(A)).$

Recall that, an approximate lattice $\Lambda$ in a connected Lie group $G$ arising as a model set over a cut-and-project scheme $(G,H,\Gamma)$ means there exists a connected locally compact group $H$ and a lattice $\Gamma\subset G\times H$, and a window $W\subset H$ such that $\Lambda=P_0(\Gamma, W)$. Proposition \ref{proposition-huruska} ensures that, for a non-compact, centerless, linear semi-simple Lie group, there is a {\it semisimple} Lie group $H$ such that any approximate Lattice in $G$ commensurate to $\Lambda'=P_0(\Gamma, W)$ over a cut-and-project scheme $(G, H, \Gamma)$. The result is essentially by Huruska \cite[Corollary 6.11]{h22}; we state it here in a form suited to our purposes and include a proof for completeness.

\begin{proposition}\label{proposition-huruska}
     Let $G$ be a non-compact, centerless, linear semi-simple Lie group and $\Lambda$ be an approximate lattice. Then there exists a semisimple Lie group $H$ and a lattice $\Gamma \leq G \times H$ such that $\Lambda$ is commensurable to an approximate lattice $\Lambda'$ over the cut-and-project scheme $(G,H,\Gamma)$.
\end{proposition}

\begin{proof}
    The proof follows easily from  \cite[Corollary 6.11]{h22}. By \cite[Corollary 6.11]{h22}, there exists a decomposition $G = G_1 \times G_2$, and $\Lambda$ is commensurable to $\Lambda_1 \times \Lambda_2'$ where $\Lambda_1 \leq  \langle\Lambda \rangle$ is an approximate lattice of $G_1$ and $\Lambda_2'$ is a lattice of $G_2$, and there exists a semisimple real Lie group $H$ and a lattice $\Gamma \leq G_1 \times H$ with cut-and-project sets commensurable to $\Lambda_1$. Suppose $W$ is the window such that $\pi_{G_1}(\Gamma \cap (G_1\times W))$ is commensurable to $\Lambda_1$.  Since $\Gamma \times \Lambda_2$ is a lattice in $G\times H$, then the same window gives that $\pi_{G}((\Gamma \times \Lambda_2)\cap (G\times W))$ is commensurable to $\Lambda$. 
\end{proof}

Proposition \ref{proposition-huruska} ensures that the hypothesis in Theorem \ref{theorem-moore’s ergodicity-type theorem for hull} is not very restrictive. See Section 2 for an example that satisfies the assumptions of Theorem \ref{theorem-moore’s ergodicity-type theorem for hull}.

%\begin{theorem}[Moore’s Ergodicity-type theorem for hull of a strong approximate lattice]\label{theorem-moore’s ergodicity-type theorem for hull}
 %   Let $G$ be a linear semisimple Lie group without a compact factor and let $\Lambda = P_0(\Gamma, W_0)$ be a regular model set over a cut-and-project scheme $(G,H,\Gamma)$, where $H$ is a linear semisimple Lie group without a compact factor and $\Gamma$ is an irreducible lattice in $G\times H$. Let $\Omega_{\Lambda}$ be the hull of $\Lambda$. 
  %  If $M$ is a non-compact subgroup of $G$, then it acts ergodically on the hull $\Omega_{\Lambda}$.
%\end{theorem}

\begin{proof}[Proof of Theorem \ref{theorem-moore’s ergodicity-type theorem for hull}]
Let $Z$ be a measurable subset of the punctured hull 
$
\Omega_{\Lambda}^*=\Omega_{\Lambda}\setminus\{\varnothing\}
$
equipped with the unique $G$-invariant measure $\nu$, defined by
$\nu(Z):=\mu\big(\beta(Z\cap X^{ns})\big),$ where the notation is as in Theorem~\ref{uniquemeasueonhull}. Assume that $M.Z\subset Z$. We first show that
$M.(Z\cap X^{ns})\subset Z\cap X^{ns}.$
Let $m\in M$ and $Q\in Z\cap X^{ns}$. By assumption, $mQ\in Z$. It remains to show that $mQ\in X^{ns}$. Since $Q\in X^{ns}$, we can write
$
\beta(Q)=(g_1,h_1)\Gamma \in Y^{ns}
$
for some $(g_1,h_1)\in G\times H$, where $h_1^{-1}W_0$ is $\Gamma$-regular. Since $M$ is a subgroup of $G$, we have
$
m\cdot (g_1,h_1)=(mg_1,h_1).
$
Using the $G$-equivariance of $\beta$, we obtain
$\beta(mQ)=m\beta(Q)=(mg_1,h_1)\Gamma.$
Since $h_1^{-1}W_0$ is $\Gamma$-regular, it follows that $\beta(mQ)\in Y^{ns}$ and hence $mQ\in X^{ns}$. Therefore,
$M.(Z\cap X^{ns})\subset Z\cap X^{ns}.$
Since $\mu(Y\setminus Y^{ns})=0$ (see \cite[Lemma 3.6]{BHP}), we have
$
\mu(\beta(Z))=\mu(\beta(Z)\cap Y^{ns})
=\mu\big(\beta(Z\cap X^{ns})\big).
$
Moreover, by the $G$-equivariance of $\beta$ and the inclusion above,
\[
\beta\big(m.(Z\cap X^{ns})\big)
= m.\beta(Z\cap X^{ns})
\subset \beta(Z\cap X^{ns}), \; \forall \; m \in M.
\]
Thus the measurable subset $\beta(Z\cap X^{ns})$ of the homogeneous space
$
Y:=(G\times H)/\Gamma
$
is invariant under the action of $M$.

Consider the characteristic function $1_{\beta(Z\cap X^{ns})}$. For $g\in G$,
$$
(g\cdot 1_{\beta(Z\cap X^{ns})})(y)
=1_{\beta(Z\cap X^{ns})}(g^{-1}y)
=1_{g(\beta(Z\cap X^{ns}))}(y).
$$
Since $\mu$ is $G$-invariant, we have
$
\mu\big(\beta(Z\cap X^{ns})\big)
=
\mu\big(m.\beta(Z\cap X^{ns})\big)
$ for all $m \in M$.
Together with the inclusion $m.\beta(Z\cap X^{ns})\subset \beta(Z\cap X^{ns})$, this implies that, for all $m \in M$,
\[
m\cdot 1_{\beta(Z\cap X^{ns})}
=
1_{\beta(Z\cap X^{ns})}
\quad \text{almost everywhere}.
\]
Since $M$ is a non-compact subgroup, Theorem~\ref{moores ergodicity} implies that the action of $M$ on $Y=(G\times H)/\Gamma$ is ergodic. Hence the equality
$M\cdot 1_{\beta(Z\cap X^{ns})}
=
1_{\beta(Z\cap X^{ns})}$
almost everywhere implies that $1_{\beta(Z\cap X^{ns})}$ is almost everywhere constant. Therefore, it takes the value either $0$ or $1$ almost everywhere, which means that
$\mu\big(\beta(Z\cap X^{ns})\big)$
is either $0$ or $1$. Consequently, $\nu(Z)=\mu(\beta(Z\cap X^{ns}))$ is either $0$ or $1$. 
\end{proof}

\section{Strong Approximate Lattices and Regular elements}

In this section, we prove Theorem \ref{theorem-regular element-strong approximate lattice}. A key ingredient is Lemma \ref{lemma-two topology}, which captures the central idea of our approach: while in the classical setting a connected Lie group $G$ with a lattice $\Gamma$, $G$ acts on the homogeneous space $G/\Gamma$, in our situation, when $\Lambda$ is a strong approximate lattice, the essential point is that $G$ instead acts on the hull $\Omega_\Lambda$, endowed with the Chabauty–Fell topology. By making use of the Chabauty–Fell topology on the hull together with Poincaré recurrence, we show that for every $a \in G$ and every identity neighborhood $U \subset G$, one has $Ua^lU \cap \Lambda \neq \varnothing$ for infinitely many integers $l$, and this lemma is the crucial step in the proof of Theorem \ref{theorem-regular element-strong approximate lattice}. First, we recall the Poincaré recurrence theorem (see \cite[Theorem 2.11]{EM11}). 

\begin{theorem}[Poincaré recurrence]\label{theorem-poincare recurrence}
Let $T: X \to X$ be a measure-preserving transformation on a probability space
$(X,\mathcal{B},\mu)$, and let $E \subseteq X$ be a measurable set. Then almost every
point $x \in E$ returns to $E$ infinitely often. That is, there exists a measurable
subset $F \subseteq E$ with $\mu(F) = \mu(E)$ such that for every $x \in F$ there exist
integers $0 < n_1 < n_2 < \cdots$ with
$T^{n_i}(x) \in E \quad \text{for all } i \ge 1.$
\end{theorem}

\begin{lemma}\label{lemma-two topology}
Let $G$ be a connected Lie group. Let $U$ be a symmetric open neighbourhood of the identity element in $G$. Let $\Lambda$ be a strong approximate Lattice, and $\Omega_{\Lambda}$ be its hull. Then, for any $a\in G$, $\Lambda^2$ intersects $Ua^lU$ for infinitely many $l$.
\end{lemma}
\begin{proof}
For given $U$, define the basic open set $\mathcal{O}_U := \{ Q \in \mathcal{C}(G) \mid Q \cap U \neq \varnothing \}$ of $\mathcal{C}(G)$ in Chabauty-Fell topology. Further define $\mathcal{O}_U(\Omega_{\Lambda})= \Omega_{\Lambda} \cap \mathcal{O}_U:= \{ Q \in \Omega_{\Lambda} \mid Q \cap U \neq \varnothing \} $ a non-empty (as $e \in \Lambda \cap U $) open set  of the hull $\Omega_{\Lambda}$ in the subspace topology. The hull $\Omega_{\Lambda}$ has a $G$-invariant probability measure, say $\mu$, as $\Lambda$ is a strong approximate lattice. A map $T_a: \Omega_{\Lambda} \to \Omega_{\Lambda} $, defined as $T_a(Q)= aQ$ for $Q \in \Omega_{\Lambda}$ (left translation), is a measure preserving transformation on probability space $(\Omega_{\Lambda}, \mu)$. By apply Theorem \ref{theorem-poincare recurrence} on measurable set $\mathcal{O}_U(\Omega_{\Lambda})$ of $\Omega_{\Lambda}$, we get that $a^lQ$ intersects $\mathcal{O}_U(\Omega_{\Lambda})$ for infinity many positive integers $l$ and almost every element $Q \in \mathcal{O}_U(\Omega_{\Lambda})$. 

For $Q \in \mathcal{O}_U(\Omega_{\Lambda})$ , there exists $u \in U$ such that $u \in Q$. Since $a^lQ$ intersects $\mathcal{O}_U(\Omega_{\Lambda})$, by definition we have $a^lQ \cap U \neq \phi$. let $u' \in a^lQ \cap U$, then $a^{-l}u' \in Q$. Now, since $Q \in \Omega_{\Lambda}$ there exists a sequence $\{g_n\Lambda\}$ that converges to $Q$ in Chabauty-Fell topology, where $g_n \in G$. By the convergence in Chabauty-Fell topology, there exists a sequences $\{g_n \lambda_{n,1}\}$ that converges to $u \in Q$ and $\{g_n \lambda_{n,2}\}$ that converges to $a^{-l}u' \in Q$, where $\lambda_{n,i}$ for $i=1,2$ belongs to $\Lambda$. Now, consider the sequence $\{\lambda^{-1}_{n,2} g_n^{-1} g_n \lambda_{n,1}\}$ that converges to $u'^{-1}a^lu$. Therefore, we have that $u'^{-1}a^lu \in \overline{\Lambda^{-1}\Lambda}$. Since $\Lambda=\Lambda^{-1}$ and $\Lambda$ is discrete, we have $u'^{-1}a^lu \in \Lambda^{2}$. Using the fact that $U$ is symmetric, we get the conclusion.
\end{proof}

%\begin{theorem}\label{theorem-regular element-strong approximate lattice}
 %Let $G$ be a linear semisimple Lie group without a compact factor. Let $\Lambda$ be an irreducible strong approximate Lattice in $G$. Then $\Lambda^2$ contains a (hyper) regular element.   
%\end{theorem}

%Theorem \ref{theorem-regular element-strong approximate lattice} is now proved using Lemma \ref{lemma-two topology}.

\begin{proof}[Proof of Theorem \ref{theorem-regular element-strong approximate lattice}]
Without loss of generality, we may assume that the center of $G$, $Z(G)$, is trivial because $Z(G)$ is contained in the kernel of the adjoint representation.
%{\color{red} why? Indeed, for $z \in Z(G)$ and $X \in \mathfrak{g}, Ad(z)X=\frac{d}{dt}|_{t=0}z\exp{(tX)}z^{-1}=X$}. 
Let $\mathfrak{g}=\mathfrak{k} + \mathfrak{p}$ be the Cartan decomposition of the Lie algebra $\mathfrak{g}$ of $G$. Assume $\mathfrak{a} \subset \mathfrak{p}$ is a maximal split torus contained in a Cartan subalgebra $\mathfrak{c}$. Let $A$ be the subgroup corresponding to $\mathfrak{a}$ and $C$ be the Cartan subgroup corresponding to $\mathfrak{c}$ of $G$. We know that for every $a\in A$, Ad$a$ is semisimple and has positive and real eigenvalues. Let us fix a regular element $a \in A$.  Let $\Delta$ be the set of nontrivial characters on $C$ occurring in the adjoint representation on $\mathfrak{g}$. 
Since $a \in A$ is regular, we have $\chi(a) \neq 1$ for every $\chi \in \Delta$. It follows that $a^l$ is regular for every nonzero integer $l$, 
because $\chi(a^l) = \chi(a)^l \neq 1$ for all $\chi \in \Delta$.
Also, for any $\chi \in \Delta$, the set $\{\chi(a^l): -\infty< l < \infty\}$ is closed and discrete subgroup of multiplicative group $\mathbb{R}^*$. Indeed, either $|\chi(a^l)|< 1$ or $|\chi(a^l)|> 1$, so for $l \to \infty$ or $l \to -\infty$, $\chi(a^l)$ $\to$ $\infty$ or $0$. Since $0 \notin \mathbb{R}^*$, the set is discrete. So, there exists a neighborhood $V$ of the identity in $A$ such that $\bigcup_{l \neq 0}Va^l$ entirely consists of regular elements. Indeed, for $v \in V$, $\chi(v)$ is close to $1$.  
Since $\chi(va^l) = \chi(v)\,\chi(a^l)$, it follows that $\chi(va^l)$ is close to $\chi(a^l)$.  
By the discreteness of the set $\{\chi(a^l) : l \in \mathbb{Z}\}$ in $\mathbb{R}^*$, we then have $\chi(va^l) \neq 1$ for all $\chi \in \Delta$.  
Hence, every element in $\bigcup_{l \neq 0} V a^l$ is regular.
By \cite[Proposition 2.2]{pr}, there exists a neighborhood $U$ of the identity in $G$ such that $Ua^lU$ is contained in the set $\bigcup_{g \in G}gVa^lg^{-1}$.
By Lemma \ref{lemma-two topology}, it follows that $\Lambda^2$ intersects $\bigcup_{g \in G}gVa^lg^{-1}$. Therefore, $\Lambda^2$ contains a regular element. 
    \end{proof}

    As an immediate consequence of Theorem \ref{theorem-regular element-strong approximate lattice}, we have the following result.

\begin{corollary}
    Let $G$ be a linear semisimple Lie group without a compact factor. Let $\Lambda$ be an irreducible strong approximate Lattice in $G$. Let $C$ be a split Cartan subgroup of $G$ and $C^0$ be its connected component of the identity. Then $\Lambda^2\cap gC^0g^{-1}\neq \emptyset$ for some $g\in G$. Moreover, if $\Lambda$ is uniform, then $\Lambda^2 \cap gCg^{-1}$ is an uniform approximate lattice in $gCg^{-1}$. 
\end{corollary}

\begin{proof}
    Without loss of generality, we may assume that the center of $G$, $Z(G)$, is trivial because $Z(G)$ is contained in the kernel of the adjoint representation. Choose a regular element $a \in C^0$. Then, by Theorem 1.1, $\Lambda^2$ contains a regular element and the regular element belongs to $\bigcup_{g \in G}gVa^lg^{-1}$, where $V$ is a neighbourhood of the identity in $C$. Since the neighbourhood $V$ is connected, choose $V \subset C^0$. Therefore, we have  $\Lambda^2\cap gC^0g^{-1}\neq \emptyset$ for some $g\in G$. Let $\lambda \in \Lambda^2\cap gC^0g^{-1}\neq \emptyset$ be a regular element. Then, the centralizer of $\lambda$ is the Cartan subgroup $gCg^{-1}.$ In case, when $\Lambda$ is uniform, it follows from \cite[Corollary 6.6]{M25} that $\Lambda^2 \cap gCg^{-1}$ is an uniform approximate lattice in $gCg^{-1}$.
\end{proof}

Let $\Lambda$ be a strong approximate lattice in a semisimple Lie group. In Example \ref{example-regular element}, we explicitly construct a regular element in $\Lambda^2$, which is not contained in $\Lambda$, in support of Theorem \ref{theorem-regular element-strong approximate lattice}.
For this, we need the following preparation. Let $\Lambda$ be as in Example \ref{example-strong app.lattice}.
Consider a left invariant metric $d$ on $\SL_2(\R)$ defined by $\langle X,Y\rangle= {\rm Tr}(XY^t)$ on $T_e(\SL_2(\R))$, tangent space of the identity, where ${\rm Tr}$ denotes trace and $Y^t$ denotes transpose of $Y$. Consider the associated  exponential map        $\exp:T_e(\SL_2(\R)) \to \SL_2(\R), $ which is always locally diffeomorphic. Let $r$ be the injectivity radius. 
Let  $B= \diag(x, -x) \in T_e(\SL_2(\R))$ be a diagonal matrix and $g= \exp(B)$. An easy computation shows that $\langle B,B^t\rangle= 2x^2$, that is, $||B||=\sqrt{2}|x|$ and $d(e,g) \leq \sqrt{2}|x|.$ Similarly, we have $\langle 2B,2B^t\rangle= 8x^2$ and $d(e,g^2) \leq 2\sqrt{2}|x|.$ 
Now, we find an element $h \in \SL_2(\R)$ such that $d(e,h)< r$, the injectivity radius, but $d(e,h^2)>r$. For this, we look for the matrix $g$ and we choose $x$ such that $\sqrt{2}|x|< r < 2\sqrt{2}|x|$.
    Similarly, if we fix any $\epsilon \leq r$, then we can find an element $g \in SL_2(\R)$ such that $d(e,g)< \epsilon$ but $d(e,g^2)> \epsilon^2$.

\begin{example}\label{example-regular element}
    Consider the scheme $(G,H,\Gamma) = \bigl(\mathrm{SL}_2(\mathbb{R}),\, \mathrm{SL}_2(\mathbb{R}),\,
\mathrm{SL}_2(\mathbb{Z}[\sqrt{2}])\bigr)$.  Let $\tau \colon \Gamma_G \to H$ be the associated $\tau$-map. Let  $W_0 := \overline{B_\varepsilon(e)} \subset H$ be a regular window  and $\Lambda = P_0(\Gamma,W_0) = \tau^{-1}(W_0)$ as in Example \ref{example-strong app.lattice}. Then, $\Lambda^2$ contains a regular element.
Indeed, note that for $g \in \Gamma_{G}$, $\tau(g)$ is the image of $g$ under non-trivial real Galois embedding of $\mathbb{\Q}(\sqrt{2})$. Hence, $\tau(g^2)=\tau(g)^2$.
Using the above details, we find a regular element $g \in \Lambda$ such that $g^2 \notin \Lambda$ but $g^2 \in \Lambda^2$ is a regular element. We can choose a diagonal matrices $g \in \Lambda$ with distinct eigenvalues such that $\tau(g) \in W_0$ (for example $\diag(\exp(x),\exp(-x))$) but $\tau(g)^2 \notin W_0$. Note that $(g,\tau(g))\in \Gamma$ implies  $(g,\tau(g))^2=(g^2, \tau(g)^2) \in \Gamma$ and $(g^2,\tau(g^2))=(g^2,\tau(g)^2)$ as  $\tau(g)^2=\tau(g^2)$. Therefore, by construction of $g$, we have $g^2 \notin \Lambda$  as $g^2 \in \Gamma_G$ but $\tau(g^2) \notin W_0$.
\end{example}

\section{Uniform Approximate Lattices arising as  Regular Model Sets and Non-Compact Cartan Subgroups}
In this section, we prove Theorem \ref{theorem-non-compact-cartan} and Corollary \ref{cor-free-abelian-group}. In a similar spirit to Theorem \ref{theorem-regular element-strong approximate lattice}, we prove Proposition \ref{proposition-R-regular-element} for the existence of an $\mathbb R$-regular element in $\Lambda^2$ for a strong approximate lattice $\Lambda$ arising as a regular model set in $G$.

%\begin{theorem}\label{theorem-non-compact-cartan}
 % Let $G$ be a linear semisimple Lie group without a compact factor and let $\Lambda = P_0(\Gamma, W_0)$ be a regular model set over a cut-and-project scheme $(G,H,\Gamma)$, where $H$ is a linear semisimple Lie group without a compact factor and $\Gamma$ is an irreducible lattice in $G\times H$. Let $C$ be a non-compact Cartan of $G$ and $C^0$ be its connected component of the identity.
 %  Then $\Lambda^2\cap gC^0g^{-1}\neq \emptyset$ for some $g\in G$. Moreover, if $\Gamma$ is uniform, then $\Lambda^2 \cap gCg^{-1}$ is an uniform approximate lattice in $gCg^{-1}$. % In particular, $\Lambda^2 \cap gC^0g^{-1}$ is a model set.
%\end{theorem}
\begin{proof}[Proof of Theorem \ref{theorem-non-compact-cartan}]
     $(1):$ Given that the Cartan subgroup $C$ is non-compact, that is, the connected component of the identity $C^0$ is a direct product of a compact torus $T$, say, and a non-trivial vector group $A$. Here, also, we can assume that $Z(G)$ is trivial. Let $\Delta$ be the set of nontrivial characters on $C$ occurring in the adjoint representation on $\mathfrak{g}$. So, an element $c \in C$ is regular if $\chi(c)\neq \pm 1$ for every $\chi \in \Delta$. Now, we choose $a \in A$ such that $\chi(a)\neq 1$ for every $\chi \in \Delta$ whose restriction to $A$ is non-trivial. Note that $\chi(a)\neq -1$ as $A$ is connected. Also, we choose $t \in T$ such that $\chi(t)\neq \pm1$ for every $\chi \in \Delta$ whose restriction to $A$ is trivial. Then $ta^l$ is regular for every non-zero integer $l$. There exists an open neighbourhood $V$ of the identity in $C^0$ such that the set $\bigcup_{l\neq 0}ta^lV$ consists of entirely regular elements. Let $U$ be an open neighbourhood of the identity element in $G$ such that for $l\neq 0$, $Uta^lU \subset \bigcup_{g \in G}gVta^lg^{-1}$ (cf. \cite[Proposition 2.2]{pr}), 

    Without loss of generality, we may assume that $U$ is a symmetric open neighbourhood of the identity element in $G$. Consider a basis open set $\mathcal{O}_U$ in Chabauty-Fell topology and define a non-empty open subset $\mathcal{O}_U(\Omega_{\Lambda})= \Omega_{\Lambda} \cap \mathcal{O}_U:= \{ Q \in \Omega_{\Lambda} \mid Q \cap U \neq \varnothing \} $ in the hull $\Omega_{\Lambda}$. Similarly, define a non-empty open subset $\mathcal{O}_{t^{-1}U}(\Omega_{\Lambda})$  in $\Omega_{\Lambda}$ corresponding to open subset $t^{-1}U$ in $G$. Consider the subgroup $M$ of $G$ generated $a$. The subgroup $M$ is non-compact, hence acts ergodically on the hull $\Omega_{\Lambda}$ by Theorem \ref{theorem-moore’s ergodicity-type theorem for hull}. That is, for infinity many non-zero integers $l$, $a^l \mathcal{O}_U(\Omega_{\Lambda})$ intersects   $\mathcal{O}_{t^{-1}U}(\Omega_{\Lambda})$.  Let $Q \in \mathcal{O}_U(\Omega_{\Lambda})$ such that $a^lQ$ intersects $\mathcal{O}_{t^{-1}U}(\Omega_{\Lambda})$. By definition, we have $a^lQ \cap t^{-1}U \neq \phi$. This implies that there exists $u' \in U$ such that $t^{-1}u' \in a^lQ\cap t^{-1}U$, that is, $a^{-l}t^{-1}u' \in Q$. Since $Q \in \mathcal{O}_U(\Omega_{\Lambda})$, there exists $u \in U$ such that $u \in Q$.

     Now, since $Q \in \Omega_{\Lambda}$ there exists a sequences $\{g_n\Lambda\}$ that converges to $Q$ in Chabauty-Fell topology, where $g_n \in G$. By the convergence in Chabauty-Fell topology, there exists a sequences $\{g_n \lambda_{n,1}\}$ that converges to $u \in Q$ and $\{g_n \lambda_{n,2}\}$ that converges to $a^{-l}t^{-1}u' \in Q$, where $\lambda_{n,i}$ for $i=1,2$ belongs to $\Lambda$. Now, consider the sequence $\{\lambda^{-1}_{n,2} g_n^{-1} g_n \lambda_{n,1}\}$ that converges to $u'^{-1}ta^lu$. Therefore, we have that $u'^{-1}a^lu \in \overline{\Lambda^{-1}\Lambda}$. Since $\Lambda=\Lambda^{-1}$ and $\Lambda$ is uniformly discrete, we have $u'^{-1}ta^lu \in \Lambda^{2}$. Using the fact that $U$ is symmetric, we get the conclusion. This proves $(1)$.

     $(2):$ If $\Gamma$ is uniform, it follows from \cite[Corollary 6.6]{M25} that $\Lambda^2 \cap gCg^{-1}$ is an uniform approximate lattice in $gCg^{-1}$. 

$(3):$ From $(1)$, there exists a regular element $\gamma \in \Lambda^2 \cap gCg^{-1}$ for some $g \in G$. Therefore, the Cartan subgroup $gCg^{-1}=C_G(\gamma)$, the centralizer of $\gamma$ in $G$. Here, we show that $\Lambda^2 \cap C_G^0(\gamma)$ is a uniform approximate lattice in $C_G^0(\gamma)$, connected component of the identity. It is well known that the index of $C_G^0(\gamma)$ in $C_G(\gamma)$ is finite. Since $\Lambda$ is a uniform approximate lattice, it is enough to show that $\rho_0(\Lambda)$ is locally finite, where $\rho_0:G \to G/C_G^0(\gamma)$ is defined by $g \mapsto gC^0(\gamma)$, that is, $\rho_0(\Lambda) \cap K$ is finite for all compact subset $K \subset G/ C_G^0(\gamma)$ (see \cite[Proposition 6.3]{M25}). It is known from \cite[Corollary 6.6]{M25} that $\rho(\Lambda)$ is locally finite where $\rho:G \to G/ C_G(\gamma).$ Consider the natural continuous map $\pi: G/ C_G^0(\gamma) \to G/ C_G(\gamma)$, then $\rho=\pi \circ \rho_0$. Let   $K \subset G/ C_G^0(\gamma)$ be compact, then $\pi(K)$ is compact in $ G/ C_G(\gamma)$. Since $\pi(K)\cap \rho(\Lambda)$ is finite and $K \cap \rho_0(\Lambda) \subset \pi^{-1}(\pi(K)\cap \rho(\Lambda))$, the subset $K \cap \rho_0(\Lambda)$ is finite. Indeed, each point of $\pi(K)\cap \rho(\Lambda)$ has almost $[C_G(\gamma):C_G^0(\gamma)]$ preimages. This proves  that $\rho_0(\Lambda)$ is locally finite and hence $\Lambda^2 \cap gC^0g^{-1}$ is an uniform approximate lattice in $gC^0g^{-1}$. The rest of the proof follows from \cite[Corollary 1.2]{M20}.
\end{proof}

The preceding result suggests us to ask the following questions:

\begin{question}
Let $G$, $\Lambda$, and $C$ be as in Theorem \ref{theorem-non-compact-cartan}.
\begin{enumerate}
    \item  Is $\Lambda^2 \cap gC^0g^{-1}$ arising as a regular model set?

    \item Is $\Lambda^2 \cap gCg^{-1}$ Mayer set?
    \end{enumerate}
\end{question}

%\begin{corollary}\label{cor-free-abelian-group}
%  Let $G$ be a linear semisimple Lie group (with compact factors). Let $\Lambda$ be a strong (as well as uniform) approximate Lattice arising as a regular model set. Then $\langle\Lambda^2 \rangle$ contains an abelian subgroup of rank $r=\mathbb R-{\rm rank}(G)$.
%\end{corollary}

As a consequence of Theorem \ref{theorem-non-compact-cartan}, we obtain Corollary \ref{cor-free-abelian-group}.

\begin{proof}[Proof of Corollary \ref{cor-free-abelian-group}]
  Let $G$ be a linear semisimple Lie group. Then there exists a Cartan subgroup $C$ such that $C^0$ is a direct product of a compact torus and a vector group of rank $r$. By Theorem \ref{theorem-non-compact-cartan}, there exists $g\in G$ such that $gC^0g^{-1}\cap\Lambda^2$ is a uniform approximate lattice of $gC^0g^{-1}.$ Therefore, there is a compact set $K$ such that $gC^0g^{-1}=K(gC^0g^{-1}\cap\Lambda^2)$. As  $gC^0g^{-1}\cap\Lambda^2\subset gC^0g^{-1}\cap\langle\Lambda^2 \rangle$, there exists a compact set $K_0$ such that $gC^0g^{-1}=K_0(gC^0g^{-1}\cap \langle\Lambda^2 \rangle)$. This implies that $gC^0g^{-1}\cap\langle\Lambda^2 \rangle$ contains a free a abelian group of rank $r=\mathbb R$-${\rm rank}(G)$. This proves the result.
  \end{proof}
  
  %Therefore, $gC^0g^{-1}/(gC^0g^{-1}\cap\langle\Lambda^2 \rangle)$ is compact, and hence $gC^0g^{-1}\cap\langle\Lambda^2 \rangle$ contains vector {\color{red} what is meaning of vector?} abelian group of full rank. In particular, this implies $\langle\Lambda^2 \rangle$ contains an abelian subgroup of rank $r=\mathbb R-{\rm rank}(G)$.

Adapting the arguments from the proofs of Theorems \ref{theorem-regular element-strong approximate lattice} and \ref{theorem-non-compact-cartan}, we deduce that, for $G$ and $\Lambda$ as in Theorem \ref{theorem-non-compact-cartan}, the set $\Lambda^2$ contains $\mathbb{R}$-regular elements in Proposition \ref{proposition-R-regular-element}. It will play a crucial role in the next section and is also of independent interest.
\begin{proposition}\label{proposition-R-regular-element}
    Let $G$ and $\Lambda$ be as in Theorem \ref{theorem-non-compact-cartan}. Then $\Lambda^2$ contains  $\mathbb{R}$-regular elements.
\end{proposition}
    
\begin{proof}
  Without loss of generality, we assume that the center of $G$ is trivial.
 Let $\mathfrak{g}=\mathfrak{k} + \mathfrak{p}$ be the Cartan decomposition of the Lie algebra $\mathfrak{g}$ of $G$. Assume $\mathfrak{a} \subset \mathfrak{p}$ is a maximal split torus contained in a Cartan subalgebra $\mathfrak{c}=\mathfrak{a}+\mathfrak{t}$, where $\mathfrak{t}\subset \mathfrak{g}^0\cap \mathfrak{k}$. Let $A$ and $T$ be the subgroups corresponding to $\mathfrak{a}$ and $\mathfrak{t}$, respectively, and $C$ be the Cartan subgroup corresponding to $\mathfrak{c}$ of $G$. We know that for every $a\in A$, Ad$a$ is semisimple and has positive and real eigenvalues, and eigenvalues of Ad$t$ have modulus one for $t \in T$. Let us fix a $\mathbb{R}$-regular element $a \in A$ and choose $t \in T$ such that $ta$ is $\mathbb{R}$-regular. Also, note that $ta^l$ is $\mathbb{R}$-regular for every non-zero integer $l$. It is known that the set of $\mathbb{R}$-regular elements of $C$ is the complement in $C$ of kernels of a finite set $S$ of positive-real valued characters on $C$.  The set $\{\chi(ta^l): -\infty< l < \infty\}$ is closed and discrete subgroup of multiplicative group $\mathbb{R}^*$, for any $\chi \in S$ as every $\chi \in S$ restricted to $T$ is trivial. So there exists a neighbourhood $V$ of the identity in $C$ such that $\bigcup_{l\neq 0}ta^lV$ consists entirely $\mathbb{R}$-regular elements.  Now, by \cite[Proposition 2.2]{pr}, there exists an open (symmetric) neighbourhood $U$ of the identity in $G$ such that $l\neq 0$, $Uta^lU \subset \bigcup_{g \in G}gVta^lg^{-1}$.
    
    Now, following the proof of Theorem \ref{theorem-non-compact-cartan} (see the second and third paragraphs), we get that $u^{-1}ta^lu \in \Lambda^2$ for some $u \in U$. Since, $u^{-1}ta^lu$ is $\mathbb{R}$-regular, it proves the statement.
\end{proof}

\section{Determination of the $\mathbb{R}$-Rank of the Ambient Group by Uniform Approximate Lattices arising as a regular model set}

In this section, we prove Theorem \ref{theorem-determine-rank}. We first establish several technical results concerning $\mathbb{R}$-regular elements and their interaction with strong approximate lattices arising as regular model sets. The underlying ideas are inspired by the work of \cite{pr}, where these results were originally developed for lattices; however, in our setting, we require suitable modifications. We begin by recalling \cite[Theorem 3.5]{pr}.

%The aim of this section is to prove Theorem 1.5. To this end, we recall several results from \cite[Section 3]{pr}  and adapt them to the setting of strong approximate lattices, where they were originally established for lattices. These adaptations rely on the same underlying ideas used in the lattice case. As these results may be of independent interest, we present their statements here in the context of strong approximate lattices and provide proofs following arguments similar to those of Prasad and Raghunathan for lattices. We begin this section by recalling \cite[Theorem 3.5]{pr}.

\begin{theorem}\label{theorem-R-regular-Raghunathan}
    Let $g \in G$ be an $\mathbb{R}$-regular element. Then there is a Zariski open neighbourhood $U_g$ of the identity in $G$ such that for $h \in U_g$, $hg^m$ (and hence $g^mh=g^m(hg^m)g^{-m}$) is $\mathbb R$-regular for large integers $m$.
\end{theorem}

We now prove the following lemma using the Borel density theorem for strong approximate lattices established in \cite[Theorem 3]{bhs19}.

%We now state a lemma for strong approximate lattices analogous to \cite[Lemma 3.7]{pr} and provide a proof based on the Borel density theorem for strong approximate lattices.

\begin{lemma}\label{lemma-generating regular and R-regular elements}
Let $G$ and $\Lambda$ be as in Theorem \ref{theorem-non-compact-cartan}. Let 
 $\{y_1,\dots,y_s\}$ be a set of finitely many elements and $\gamma$ in $\langle\Lambda^2 \rangle$ an $\mathbb{R}$-regular
element. Then there exists an element
$\delta \in \langle\Lambda^2 \rangle$ such that for infinitely many positive integers $m$,
the elements $ \gamma^{m} \delta y_i^{-1}$ are regular and $\mathbb{R}$-regular for $1 \le i \le s$.
\end{lemma}

\begin{proof}
Let $U_\gamma$ be the Zariski open neighbourhood of the identity in $G$ for the $\mathbb{R}$-regular
element $\gamma$ as in Theorem \ref{theorem-R-regular-Raghunathan}, and let $\mathcal{S}$ be the set of all regular
elements in $G$. The set $\mathcal{S}$ is a non-empty Zariski open subset
of $G$ (see \cite[Remark 1.2]{pr}).

Define the set $V = \left( \bigcap_{i=1}^s U_\gamma y_i \right)
\cap
\left( \bigcap_{i=1}^s \mathcal{S} y_i \right).$ This is a finite intersection of non-empty Zariski open subsets of $G$,
and therefore is itself a non-empty Zariski open subset.
It is known that,
 $\Lambda$ is Zariski dense in $G$ by \cite[Theorem 3]{bhs19}; hence there exists an element
$\delta \in \Lambda \subset \langle\Lambda^2 \rangle$ such that $\delta \in V$.
For sufficiently large $m$,
the elements $ \gamma^{m} \delta y_i^{-1}$ are $\mathbb{R}$-regular for all $1 \le i \le s$
(cf. Theorem 6.1).
Since for all $i$, $\delta y_i^{-1} \in \mathcal{S}$,
the set
$V' = \bigcap_{i=1}^s \mathcal{S} (\delta y_i^{-1})^{-1}$
is a Zariski open neighbourhood of the identity in $G$.
Because the Zariski open set
$\bigcup_{m=1}^{\infty} \gamma^{m} V'$ is quasi-compact, there exists a positive integer $p$ such that
$\bigcup_{m=1}^{\infty} \gamma^{m} V'
=
\bigcup_{m \leq p} \gamma^{m} V'.$
This implies that for infinitely many positive integers $m$,
$\gamma^{m} \in V',$ that is, the elements $ \gamma^{m} \delta y_i^{-1}$
are regular for all $1 \le i \le s$.
This completes the proof.
\end{proof}

\begin{proposition}\label{proposition-can't-cover}
Let $G$ and $\Lambda$ be as in Theorem \ref{theorem-non-compact-cartan}.
Let $\mathcal{R}$ be the set of $\mathbb{R}$-regular and regular elements
in $G$. Then a finite number of translates of $ \langle\Lambda^2 \rangle \setminus \mathcal{R}$
cannot cover $\langle\Lambda^2 \rangle$.
\end{proposition}

\begin{proof}
Let $\mathcal{Z} = \langle\Lambda^2 \rangle \setminus \mathcal{R}$. 
Suppose, for contradiction, that
$\langle \Lambda^2 \rangle = \bigcup_{i=1}^s \mathcal{Z}\delta_i$, with $\delta_i \in \langle \Lambda^2  \rangle.$
Let $\gamma \in \Lambda^2$ be an $\mathbb{R}$-regular element 
($\mathbb{R}$-regular elements exist in $\Lambda^2$ in view of Proposition \ref{proposition-R-regular-element}).
Choose $\delta \in \langle\Lambda^2 \rangle$ and a positive integer $m$ such that $\gamma^{m}\delta \delta_i^{-1}$ is $\mathbb{R}$-regular and regular for all $1 \le i \le s$
(cf. Lemma \ref{lemma-generating regular and R-regular elements}).
Since $\langle\Lambda^2 \rangle= \bigcup_{i=1}^s \mathcal{Z}\delta_i,$ it follows that $\gamma^{m}\delta \in \mathcal{Z}\delta_i$
for some $i$. Hence
$\gamma^{m}\delta^{-1}\delta_i^{-1} \in \mathcal{Z},$
which contradicts the fact that $\gamma^{m}\delta \delta_i^{-1}$ is
$\mathbb{R}$-regular and regular.
This contradiction completes the proof.
\end{proof}

%\begin{lemma}
%Let $K$ be a group and $H \le K$ a subgroup of finite index. Then for any subgroup $L \le K$,
%$[L : L \cap H] \le [K : H].$
%\end{lemma}

%\begin{proof}
%Let $[K:H]=m<\infty$, so that $K=\bigsqcup_{i=1}^m k_iH$ for suitable representatives $k_1,\dots,k_m\in K$. Intersecting with $L$, we obtain $L=\bigcup_{i=1}^m (L\cap k_iH).$
%If $L\cap k_iH\neq\emptyset$, choose $x\in L\cap k_iH$. Then $x=k_ih$ for some $h\in H$, and one verifies that
%$L\cap k_iH = x(L\cap H),$
%so each nonempty intersection $L\cap k_iH$ is a single left coset of $L\cap H$ in $L$. Since there are at most $m$ such intersections, $L$ is covered by at most $m$ distinct cosets of $L\cap H$, hence
%$[L:L\cap H]\le m=[K:H].$
%\end{proof}

\begin{proposition}\label{prop-P}
Let $G$ and $\Lambda$ be as in Theorem \ref{theorem-non-compact-cartan} and let $r$ be the $\mathbb{R}$-rank of $G$. Let $\gamma$ be either a regular or an $\mathbb{R}$-regular element of $\langle\Lambda^2 \rangle$. Then the following are held.
\begin{enumerate}
    \item $\langle\Lambda^2\cap C_G(\gamma)^0\rangle$ contains a free abelian subgroup of rank $k \le r$ with finite index, and 
    $C_{\langle\Lambda^2 \rangle}(\gamma)$ contains a free abelian subgroup of rank $k \le r$.  
    \item If $\Lambda$ is uniform approximate Lattice and $\gamma$ is $\mathbb{R}$-regular then $\langle\Lambda^2 \cap C_G(\gamma)^0\rangle$ contains a free abelian subgroup of rank $r$ with finite index.  
    %(in $\langle\Lambda^2 \cap C_G(\gamma)^0\rangle$). 

    \item If $\Lambda$ is uniform approximate Lattice and $\gamma$ is $\mathbb{R}$-regular then $C_{\langle\Lambda^2 \rangle}(\gamma)$ contains a free abelian subgroup of rank $r$ with compact quotient in $C_{\langle\Lambda^2 \rangle}(\gamma)$ and $C_G(\gamma)$.
\end{enumerate}

\end{proposition}

\begin{proof}
$(1):$ Since $C_G(\gamma)$ is algebraic, $C_G(\gamma)/C_G(\gamma)^\circ$ is finite. 
If $\gamma$ is regular, then $C_G(\gamma)^\circ$, being the identity component of a Cartan subgroup, is of the form $T \times P$, where $T$ is compact and $P$ is a group isomorphic to $\mathbb{R}^k$ with $k \le r$. 
If $\gamma$ is $\mathbb{R}$-regular, then in view of \cite[Remark~1.6(2)]{pr}, $G_\gamma^\circ$ is a direct product of a compact group and a vector group of dimension $r$.

Note that, by Lemma 2.3 of \cite{M25}, $\Lambda^2\cap C_G(\gamma)^\circ$ is an approximate subgroup of $C_G(\gamma)^\circ$ and uniformly discrete by Proposition \ref{proposition-uniformly discrete-app latices from model set}. Since $C_G(\gamma)^0$ is abelian, by Proposition 5.20 of \cite{M25} , $\langle\Lambda^2\cap C_G(\gamma)^\circ\rangle$ is a finitely generated abelian group. This proves $(1).$

$(2):$ Since $\Lambda$ is uniform, $\Lambda^2\cap C_G(\gamma)^\circ$ is uniform approximate Lattice in $C_G(\gamma)^\circ$, and hence $C_G(\gamma)^\circ=K(\Lambda^2\cap C_G(\gamma)^\circ)$ for some compact subgroup $K$ of $C_G(\gamma)^\circ$. Therefore, $ C_G(\gamma)^\circ/\langle\Lambda^2 \cap C_G(\gamma)^\circ\rangle$ is compact. Further, since $\langle\Lambda^2\cap C_G(\gamma)^\circ\rangle$ is a finitely generated abelian group, $\langle\Lambda^2\cap C_G(\gamma)^\circ\rangle$ contains a free abelian subgroup isomorphic $P_\mathbb Z$ isomorphic to $\mathbb Z^k$ with $k\leq r$, if $\gamma$ is $\mathbb R$-regular, $k=r$, and also, $\langle\Lambda^2 \cap C_G(\gamma)^0\rangle/P_{\mathbb Z}$ is finite. This proves $(2).$

$(3):$ Since $\langle\Lambda^2 \cap C_G(\gamma)^\circ\rangle\subset\langle\Lambda^2 \rangle\cap C_G(\gamma)^\circ\subset C_G(\gamma)^0$, and $\langle\Lambda^2\cap C_G(\gamma)^\circ\rangle$ is a finitely generated abelian group, it follows that $C_G(\gamma)^\circ \cap \langle\Lambda^2\rangle$ contains a free abelian group $P_{\mathbb Z}$ isomorphic to $\mathbb Z^k$, $k\leq r$; in case if $\gamma$ is $\mathbb R$-regular, then $k=r$. Therefore, if $\gamma$ is $\mathbb{R}$-regular, then $C_{\langle\Lambda^2 \rangle}(\gamma)$ contains a free abelian subgroup of rank $r$. Moreover, $C_{\langle\Lambda^2 \rangle}(\gamma)/P_{\mathbb Z}$ is compact (as it is contained in $C_G(\gamma)/P_{\mathbb Z}$ isomorphic to $T\times \mathbb R^k/\mathbb Z^k$). This proves $(3)$.
\end{proof}
%$\langle\Lambda^2 \cap C_G(\gamma)^\circ\rangle\subset\langle\Lambda^2 \rangle\cap C_G(\gamma)^\circ$, and  
 %$\langle\Lambda^2 \rangle\cap C_G(\gamma)^\circ$ is a subgroup of in $C_{\langle\Lambda^2 \rangle}(\gamma)$.
% By Theorem {\color{red} ref and only intersection, not generated subgroup??}, $\langle\Lambda^2 \cap C_G(\gamma)^\circ\rangle$ is an approximate subgroup in $C_G(\gamma)^\circ$. Since $C_G(\gamma)^\circ$ is abelian, by Theorem we have $\langle\Lambda^2 \cap C_G(\gamma)^\circ\rangle$ is finitely generated abelian subgroup, and hence it contains a free abelian subgroup of rank $k\leq r.$ Therefore, $C_G(\gamma)^\circ$
%contains a free abelian subgroup of rank $k \le r$.

% Now, assume $\Lambda$ is an uniform approximate Lattice. Since $\gamma$ is $\mathbb R$-regular,  $C_G(\gamma)^\circ$ has a vector subgroup of dimension $r$. Since, $\Lambda$ is an uniform approximate Lattice, $\Lambda^2 \cap C_G(\gamma)^\circ$ is an uniform approximate lattice by Corollary 6.6 of .  In particular, we have $C_G(\gamma)^0=K(\langle\Lambda^2\rangle \cap C_G(\gamma)^\circ)$ {\color{red} only $\Lambda^2$ }, and $\langle\Lambda^2 \rangle\cap C_G(\gamma)^\circ/\langle\Lambda^2 \cap C_G(\gamma)^\circ\rangle$ {\color{red} why? from previous line right?} is compact. Therefore, if $\Lambda$ is an uniform approximate lattice, and $\gamma$ is $\mathbb R$-regular, then $C_{\langle\Lambda^2 \rangle}(\gamma)$ contains an abelian subgroup of rank $r$ with quotient is compact.

In Theorems \ref{theorem-R} and \ref{theorem-RR}, we present two intrinsic characterizations of the 
$\mathbb R$-rank of the linear semisimple Lie group $G$, in terms of uniform approximate lattices and their interaction with Cartan subgroups. The proof of Theorem \ref{theorem-determine-rank} relies on Theorem \ref{theorem-RR}. On the other hand, Theorem \ref{theorem-R} may be regarded as being of independent interest.

\begin{theorem}\label{theorem-R}
 Let $G$ and $\Lambda$ be as in Theorem \ref{theorem-non-compact-cartan}. In addition, assume that $\Lambda$ is uniform. Let $\mathcal{M}_t$ be the set of elements $\gamma\in\langle\Lambda^2\rangle$ such that $\langle \Lambda^2 \cap C_G(\gamma)^0\rangle$
has a subgroup of finite index isomorphic to a free abelian group of rank $\le t$.
Then $$\operatorname{Rank}_{\mathbb{R}}(G)
=
\min \left\{
t \;\middle|\;
\text{there exists a finite subset } I \subset \langle \Lambda^2 \rangle
\text{ such that } \mathcal{M}_t I = \langle\Lambda^2 \rangle
\right\}.$$
\end{theorem}

 \begin{proof}
Let $r = \operatorname{Rank}_{\mathbb{R}}(G)$.
By Proposition \ref{prop-P}, $\mathcal M_r$ contains every 
regular element in $\langle \Lambda^2 \rangle$.
The set of regular elements of $G$ is a nonempty
Zariski open subset of $G$. Using the quasi-compactness of the Zariski topology, it follows that finitely many translates of $\mathcal{M}_r$ cover $\langle\Lambda^2\rangle$, that is, there exists a finite set $I \subset \langle \Lambda^2 \rangle$ such that $\mathcal{M}_r I = \langle\Lambda^2 \rangle$.

Therefore, it remains to prove that, if $t<r$, then finitely
many translates of $\mathcal{M}_t$ cannot cover $\langle \Lambda^2 \rangle$. In view of Proposition \ref{prop-P}, if $\gamma \in \langle \Lambda^2 \rangle$ is
$\mathbb{R}$-regular and regular, then
the $\langle\Lambda^2\cap C_G(\gamma)^0 \rangle\subset C_{\langle\Lambda^2\rangle}(\gamma)$ contains a free abelian group of rank $r$ with compact quotient.
Since $t<r$, the element $\gamma$ does not belong to
$\mathcal{M}_t$. Hence, if $\mathcal{R}$ is the set of $\mathbb{R}$-regular and regular elements
in $G$, we have $\mathcal M_t\subset \langle \Lambda^2 \rangle \setminus \mathcal R$.
By Proposition \ref{proposition-can't-cover}, finitely many translates of $\mathcal M_t$ can't cover $\langle \Lambda^2 \rangle$.
This completes the proof.
\end{proof}

\begin{theorem}\label{theorem-RR}
 Let $G$ and $\Lambda$ be as in Theorem \ref{theorem-non-compact-cartan}. In addition, assume that $\Lambda$ is uniform. Let $\mathcal{S}_t$ be the set of elements $\langle\Lambda^2\rangle$ such that $C_{\langle \Lambda^2\rangle}(\gamma)$
has a subgroup isomorphic to a free abelian group of rank $\le t$ with compact quotient.
Then $$\operatorname{Rank}_{\mathbb{R}}(G)
=
\min \left\{
t \;\middle|\;
\text{there exists a finite subset } I \subset \langle \Lambda^2 \rangle
\text{ such that } \mathcal{S}_t I = \langle\Lambda^2 \rangle
\right\}.$$
\end{theorem}

 \begin{proof}
Let $r = \operatorname{Rank}_{\mathbb{R}}(G)$.
By Proposition \ref{prop-P}, $\mathcal S_r$ contains every 
regular element in $\langle \Lambda^2 \rangle$.
The set of regular elements of $G$ is a nonempty
Zariski open subset of $G$. Using the quasi-compactness of the Zariski topology, it follows that finitely many translates of $\mathcal{S}_r$ cover $\langle\Lambda^2\rangle$, that is, there exists a finite set $I \subset \langle \Lambda^2 \rangle$ such that $\mathcal{S}_r I = \langle\Lambda^2 \rangle$.

Therefore, it remains to prove that, if $t<r$, then finitely
many translates of $\mathcal{S}_t$ cannot cover $\langle \Lambda^2 \rangle$. In view of Proposition \ref{prop-P}, if $\gamma \in \langle \Lambda^2 \rangle$ is
$\mathbb{R}$-regular and regular, then
the $\langle\Lambda^2\cap C_G(\gamma)^0 \rangle$ contains a free abelian group of rank $r$.
Since $t<r$, the element $\gamma$ does not belong to
$\mathcal{S}_t$. Hence, if $\mathcal{R}$ is the set of $\mathbb{R}$-regular and regular elements
in $G$, we have $\mathcal S_t\subset \langle \Lambda^2 \rangle \setminus \mathcal R$.
By Proposition \ref{proposition-can't-cover}, finitely many translates of $\mathcal S_t$ can't cover $\langle \Lambda^2 \rangle$.
This completes the proof.
\end{proof}

%By Borel density theorem for strongly approximate lattice, $\langle \Lambda^2 \rangle$ is Zariski dense, the translates of this open set by
%elements of $\Gamma$ {\color{red}$\langle \Lambda^2 \rangle$} cover $G$. By quasi-compactness
%of the Zariski topology, finitely many such translates
%suffice to cover $G$. Intersecting with $\Gamma$ {\color{red}$\langle \Lambda^2 \rangle$},
%it follows that there exists a finite subset
%$I \subset \Gamma {\color{red}\langle \Lambda^2 \rangle}$ %such that
%\[
%\mathcal{S}_r I = \Gamma. {\color{red}\langle \Lambda^2 %\rangle \subset \mathcal{S}_r I}
%\]

\begin{proof}[Proof of Theorem \ref{theorem-determine-rank}]
Let $G_1$ and $G_2$ be connected linear semisimple groups without compact factors, and suppose that the subgroups $\langle \Lambda_1^2 \rangle \subset G_1$ and $\langle \Lambda_2^2 \rangle \subset G_2$ are such that
$\langle \Lambda_1^2 \rangle \cong \langle \Lambda_2^2 \rangle.$
Set $r_i = \mathbb{R}\text{-rank}(G_i)$ for $i=1,2$. By Theorem \ref{theorem-RR}, $r_i$ is the minimal integer $t$ such that there exist  finite subsets $I_i \subset \langle \Lambda_i^2 \rangle$ with
$\mathcal{S}_t(\langle \Lambda_i^2 \rangle)\,I_i = \langle \Lambda_i^2 \rangle$,
 where\[
\mathcal{S}_t(\langle \Lambda_i^2 \rangle)
=
\left\{
\begin{array}{l}
\gamma \in \langle \Lambda_i^2 \rangle \; |\;
C_{\langle\Lambda_i^2\rangle}(\gamma)^0 
\text{ contains a subgroup isomorphic to } \mathbb{Z}^t \\
\text{ \qquad  \qquad  \qquad   \qquad \; with compact quotient.}
\end{array}
\right\}.
\]
This characterization depends only on the group  $C_{\langle\Lambda_i^2\rangle}(\gamma)^0.$ Let 
$\varphi : \langle \Lambda_1^2 \rangle \to \langle \Lambda_2^2 \rangle$ be the topological group isomorphism. This implies
$C_{\langle\Lambda_1^2\rangle}(\gamma)^0$ and $C_{\langle\Lambda_1^2\rangle}(\varphi(\gamma))^0$ are isomorphic as topological groups. Therefore, we have
$\varphi\big(\mathcal{S}_t(\langle \Lambda_1^2 \rangle)\big)
=
\mathcal{S}_t(\langle \Lambda_2^2 \rangle).$
Since
$\mathcal{S}_t(\langle \Lambda_1^2 \rangle)\, I_1
=
\langle \Lambda_1^2 \rangle$
there exists a finite set $\varphi(I_1)$ such that
$\mathcal{S}_t(\langle \Lambda_2^2 \rangle)\, \varphi(I_1)
=
\langle \Lambda_2^2 \rangle$, and hence
$\mathbb{R}\text{-rank}(G_1) = \mathbb{R}\text{-rank}(G_2).$
\end{proof}

  We conclude this section by raising the following natural question.

  \vspace{.2cm}
\noindent
{\bf Remark and Question:} Let $G$ be a semisimple Lie group and $\Lambda \subset G$ an approximate subgroup. By a work of Prasad and Raghunathan \cite{pr}, the analogues of Theorems \ref{theorem-non-compact-cartan} and \ref{theorem-determine-rank} are known to hold when $\Lambda$ is a lattice, and these results extend to the setting of uniform approximate lattices arising as a regular model set in this article (under some assumption on $G$). It is therefore natural to ask to what extent these results remain valid, namely, under which conditions on $G$ or for which classes of approximate subgroups $\Lambda \subset G$ Theorems \ref{theorem-non-compact-cartan} and \ref{theorem-determine-rank} continue to hold; in particular, whether they hold for all strong approximate lattices, or more generally for arbitrary approximate lattices.

\vspace{.2cm}

\noindent
{\bf Acknowledgement:}
 The second author acknowledges the support of the NBHM Postdoctoral Fellowship (PDF no. 0204/27/(29)/2023/RD-II/11930), gratefully acknowledges IISER Mohali for institutional support, and thanks IIT Roorkee for the warm hospitality during his visit.

\vspace{.2cm}
\noindent
{\bf Conflict of interest:}
On behalf of all authors, the corresponding author states that there is no conflict of interest.

 \bibliography{References-2.bib}
\bibliographystyle{amsalpha}
\end{document}